\documentclass[a4paper,12pt,onecolumn]{article}
\usepackage{etex}

\usepackage{hyperref}
\usepackage[vmargin=2cm,hmargin=2cm,headheight=14.5pt,top=2cm,headsep=.5cm]{geometry}

\usepackage{bm}
\usepackage{empheq}
\usepackage{stackrel}
\usepackage{cases}
\usepackage{mathtools}
\usepackage{amsthm,amsmath,amscd}
\usepackage{amssymb,amsfonts}
\usepackage{makeidx}
\usepackage[all]{xy}
\usepackage{perpage}
\usepackage{tikz-cd}
\tikzcdset{every label/.append style = {font = \small}}
\usepackage[symbol]{footmisc}

\MakePerPage[2]{footnote}
\usepackage{graphicx}
\DeclareMathSizes{12}{12}{8}{6}

\usepackage{cite}
\usepackage{url}
\usepackage[charter]{mathdesign}
\usepackage{accents}

\newtheoremstyle{ptheorem}{1em}{0em}{\itshape}{}{\bfseries}{.}{.5em}{\thmname{#1}\thmnumber{ #2}\thmnote{ (\hspace{-.01pt}{#3})}}

\theoremstyle{ptheorem}

\newtheorem{thm}{Theorem}[section]

\newtheorem{lem}[thm]{Lemma}
\newtheorem{cor}[thm]{Corollary}

\newtheoremstyle{hdef}{1em}{0em}{}{}{\bfseries}{.}{.5em}{\thmname{#1}\thmnumber{ #2}\thmnote{ (\hspace{-.01pt}{#3})}}
\theoremstyle{hdef}

\newtheorem{dfn}[thm]{Definition}
\newtheorem{rem}[thm]{Remark}

\makeatletter
\newtheoremstyle{premark}{1em}{0em}{
\addtolength{\@totalleftmargin}{1.5em}
\addtolength{\linewidth}{-1.5em}
\parshape 1 1.5em \linewidth}{}{\scshape}{.}{.5em}{}
\makeatother

\theoremstyle{premark}

\newtheorem{exa}[thm]{Example}

\numberwithin{equation}{section}
\numberwithin{figure}{section}

\DeclareMathOperator{\codim}{codim}

\DeclareMathOperator{\Id}{Id}

\DeclareMathOperator{\dif}{d}

\DeclareMathOperator{\Gr}{Gr}
\newcommand{\cC}{{\mathcal C}}

\newcommand{\cM}{{\mathcal M}}

\newcommand{\bN}{{\mathbb N}}

\newcommand{\bP}{{\mathbb P}}

\newcommand{\bR}{{\mathbb R}}
\newcommand{\bS}{{\mathbb S}}

\newcommand{\bZ}{{\mathbb Z}}

\renewcommand{\a}{\alpha}
\renewcommand{\b}{\beta}
\renewcommand{\c}{\gamma}
\renewcommand{\d}{\delta}
\renewcommand{\l}{\lambda}
\newcommand{\e}{\epsilon}

\renewcommand{\phi}{\varphi}

\newcommand{\<}{\langle}
\renewcommand{\>}{\rangle}

\renewcommand{\<}{\left<}
\renewcommand{\>}{\right>}

\renewcommand{\(}{\left(}
\renewcommand{\)}{\right)}

\newcommand{\til}{\tilde}

\newcommand{\lils}[1]{\stackrel[{#1}]{}{\overline\lim}}
\newcommand{\olb}[1]{%
  \vbox{\offinterlineskip\ialign{\hfil##\hfil\cr $\rotatebox[origin=c]{90}{$]$}$\cr\noalign{\kern-.45ex}{$#1$}\cr}}}
\allowdisplaybreaks
\begin{document}
\title{A Lipschitz condition along a transversal foliation implies local uniqueness for ODEs\footnote{Supported by Xunta de Galicia (Spain), project EM2014/032 and MINECO, project MTM2017-85054-C2-1-P.}}

\author{
J. \'Angel Cid\\
\normalsize
Departamento de Ma\-te\-m\'a\-ti\-cas,\\ 
\normalsize Universidade de Vigo, Spain.\\ 
\normalsize e-mail: angelcid@uvigo.es\\
F. Adri\'an F. Tojo%\footnote{Partially supported by  AEI and FEDER (Spain) projects MTM2013-43014-P and MTM2016-75140-P.} 
\\ \normalsize
Departamento de An\'alise Ma\-te\-m\'a\-ti\-ca, Facultade de Matem\'aticas,\\ 
\normalsize Universidade de Santiago de Com\-pos\-te\-la, Spain.\\ 
\normalsize e-mail: fernandoadrian.fernandez@usc.es\\
}
\date{}

\maketitle

\begin{abstract} We prove the following result: if a continuous vector field $F$ is Lipschitz when restricted to the hypersurfaces determined by a suitable foliation and a transversal condition is satisfied at the initial condition, then $F$ determines a locally unique integral curve. We also present some illustrative examples and sufficient conditions in order to apply our main result.

\end{abstract}

\noindent {\bf Keywords:} Uniqueness; Lipschitz condition; Foliation; Modulus of continuity; Rotation formula.
 
\noindent     \textit{2010 MSC:} 34A12. 

\section{Introduction}

Uniqueness for ODEs is an important  and quite old subject, but still an active field of research \cite{cons,DiNoSi14,ferreira}, being Lipschitz uniqueness theorem the cornerstone on the topic. Besides the existence of many generalizations of that theorem, see \cite{agalak,cl,hartman}, one recent and fruitful line of research has been the searching for alternative or weaker forms of  the Lipschitz condition. For instance, let $U\subset\bR^2$ be an open neighborhood of $(t_0,x_0)$ and $f:U\subset\bR^2\to\bR$ be continuous and consider the scalar initial value problem
\begin{equation}\label{eq-ivp-scalar} x'(t)=f(t,x(t)),\, x(t_0)=x_0. \end{equation}
It was proved, independently by Mortici, \cite{mortici}, and Cid and Pouso \cite{cidpouso1,cp2}, that local uniqueness holds provided that the following conditions are satisfied:
\begin{itemize}

\item$f(t,x)$ is Lipschitz with respect to $t$,

\item$f(t_0,x_0)\not=0$.

\end{itemize}

A more general result had been proved before by Stettner  and Nowak, \cite{sn}, but in a paper restricted to German readers.  They proved that if $U\subset \bR^2$ is an open neighborhood of $(t_0,x_0)$, $f:U\subset \bR^2\to\bR$ is continuous and $(u_1,u_2)\in\bR^2$ such that 

\begin{itemize}

\item $|f(t,x)-f(t+ku_1,x+ku_2)|\le L |k|\quad \mbox{on $D$}$, 

\item $u_2\not= f(t_0,x_0) u_1$,

\end{itemize}
then the scalar problem \eqref{eq-ivp-scalar} has a unique local solution. By taking either $(u_1,u_2)=(0,1)$ or $(u_1,u_2)=(1,0)$ this result covers both the classical Lipschitz uniqueness theorem and the previous alternative version. Moreover  this result has been remarkably generalized in \cite{DiNoSi14} by Dibl{\'\i}k, Nowak and Siegmund by allowing the vector $(u_1,u_2)$ to depend on $t$. 

Let us now consider the autonomous initial value problem for a system of differential equations
\begin{equation}\label{peq-aut}
z'(t)=F(z(t)),\ z(t_0)=p_0,
\end{equation}
where $n\in\bN$, $F:U\subset\bR^{n+1}\to\bR^{n+1}$ and $p_0\in U$.

% We say that $f:U\subset \R^{n+1}\to \R^n$ is {\it Lipschitz
%continuous when fixing component} $i_0\in\{0,1,\ldots ,n\}$ if
%there exists $K>0$ such that
%\begin{eqnarray*}&& \|f(u_0 ,\ldots ,v_{i_0},\ldots ,u_n )-
%f(\bar{u}_0 ,\ldots ,v_{i_0},\ldots ,\bar{u}_n ) \|_{\infty} \le
% \\ && K \|(u_0 ,\ldots ,u_{i_0 -1},u_{i_0 +1},\ldots ,u_n )- (\bar{u}_0
% ,\ldots ,\bar{u}_{i_0 -1},\bar{u}_{i_0 +1},\ldots
%,\bar{u}_n ) \|_{\infty},\end{eqnarray*} for all $(u_0 ,\ldots
%,v_{i_0},\ldots ,u_n ), (\bar{u}_0 ,\ldots ,v_{i_0},\ldots
%,\bar{u}_n )\in U$ and $K$ is called a Lipschitz constant.
%
%We say that $f$ is {\it locally Lipschitz continuous when fixing
%component} $i_0\in\{0,1,\ldots ,n\}$ if for every $(t,x)\in U$
%there exists a neighbourhood $V\subset U$ of $(t,x)$ such that the
%restriction of $f$ to $V$ is Lipschitz continuous when fixing
%component $i_0\in\{0,1,\ldots ,n\}$.
%
%We say that $f$ is {\it (locally) Lipschitz continuous with
%respect to} $x$ if it is  (locally) Lipschitz continuous when
%fixing component $i_0=0$.
%
%It is well-known that if there exists
%$\frac{\partial{f}(t,x)}{\partial{x}}$ and it is continuous in
%$U$, then $f$ is locally Lipschitz continuous with respect to $x$.
%An analogous result, of course, is valid for the case of functions
%locally Lipschitz continuous when fixing a component $i_0$.

Trough the paper we shall need the following definition: if $g:D\subset\bR^{n+1}\to E$, where $E$ is a normed space, we will say that $g$ is {\it Lipschitz in $D$ when fixing the first variable} if there exists $L>0$ such that for all $(s,x_1,x_2,\ldots,x_n),(s,y_1,y_2,\ldots,y_n)\in D$ we have that
$$\|g(s,x_1,x_2,\ldots,x_n)-g(s,y_1,y_2,\ldots,y_n)\|_{E} \le L \|(x_1,x_2,\ldots,x_n)-(y_1,y_2,\ldots,y_n)\|,$$
and where $\|\cdot\|$ stands for any norm in $\bR^{n+1}$. Moreover, for any function $g$ with values in $\bR^{n+1}$ we denote $g=(g_1,g_2,\dots,g_{n+1})$. 

The following alternative version of Lipschiz uniqueness theorem for systems has been proved by Cid in  \cite{Cid03}.
\begin{thm} \label{cor-cid} Let $U\subset\bR^{n+1}$ an open neighborhood of $p_0$ and $F:U\to\bR^{n+1}$ continuous. If  
moreover
\begin{itemize}
\item $F$ is Lipschitz in $U$ when fixing the first variable,
\item $F_{1}(p_0)\ne0$,
\end{itemize}
then there exists $\a>0$ such that problem \eqref{peq-aut}has a unique solution in $[t_0-\a,t_0+\a]$.
\end{thm}

\begin{rem} The classical Lipschitz theorem is included in the previous one. In order to see this, let $n\in\bN$, $U\subset\bR^{n+1}$ be an open set, $f:U\to\bR^n$ and $(t_0,x_0)\in U$ and consider the non-autonomous problem
\begin{equation}\label{peq}
x'(t)=f(t,x(t)),\ x(t_0)=x_0.
\end{equation}
As it is well known, problem \eqref{peq} is equivalent to the autonomous one \eqref{peq-aut}, where \begin{displaymath}F(z_1,z_2,\ldots,z_{n+1}):=(1,f(z_1,z_2,\ldots,z_{n+1})),\end{displaymath} and $p_0:=(t_0,x_0)$. Now, if $f(t,x)$ is  Lipschitz with respect to $x$ then $F(z_1,z_2,\ldots,z_{n+1})$ is  Lipschitz  when fixing the first variable and moreover $F_{1}(p_0)=1\ne0$, so Theorem \ref{cor-cid} applies.
\end{rem}
%\begin{cor}[{\cite[Theorem 3.2]{Cid03}}] Let $U\subset\bR^{n+1}$ be an open set, $f:U\to\bR^n$ and $(t_0,x_0)\in U$. We suppose that $f$ is continuous and locally Lipschitz continuous when fixing a component $i_0\in\{0,1,\dots,n\}$. Then there exists $\a>0$ such that problem \eqref{peq} has a unique solution in $[t_0-\a,t_0+\a]$ provided that either $i_0=0$ or $f_{i_0}(t_0,x_0)\ne0$.
%\end{cor}

Recently, Dibl{\'\i}k, Nowak and Siegmund have obtained in \cite{SiNoDi16} a generalization of both \cite{Cid03} and \cite{sn}. Their result reads as follows:  

\begin{thm}\label{thDNS} Let $U\subset\bR^{n+1}$ an open neighborhood of $p_0$, $F:U\to\bR^{n+1}$ be continuous and $\cal{V}$ a linear  hyperplane in $\bR^{n+1}$ such that

\begin{itemize}
\item $F$ is Lipschitz continuous along $\cal{V}$, that is, there exists $L>0$ such that if $x,y\in U$ and $x-y\in \cal{V}$ then
$$\|F(x)-F(y)\|\le L \|x-y\| $$  
\item Transversality condition: $F(p_0)\not\in \cal{V}$.
\end{itemize}
then there exists $\a>0$ such that problem~\eqref{peq-aut} has a unique solution in $[t_0-\a,t_0+\a]$.
 
\end{thm}

The previous theorem has the following geometric meaning: uniqueness for the autonomous system \eqref{peq-aut} follows provided that the continuous vector field $F$ is Lipschitz when restricted to a family of parallel hyperplanes to $\cal{V}$ that covers $U$ and that the vector field at the initial condition $F(p_0)$ is transversal to $\cal{V}$.

 Our main goal in this paper is to extend Theorem \ref{thDNS} from the linear foliation generated by the hyperplane $\cal{V}$ to a general $n$-foliation. The paper is organized as follows: in Section 2 we present our main result which relies on an appropriate change of coordinates and Theorem \ref{cor-cid}. We will show by examples that our result is in fact a meaningful generalization of  Theorem \ref{thDNS}. In Section 3 we present some useful results about  Lipschitz functions,  including the definition of a modulus of  Lipschitz continuity along a hyperplane that will be used in Section 4 for obtaining explicit sufficient conditions on $F$ for the existence of a suitable $n$-foliation. Another key ingredient for that result shall be a general rotation formula proved too at Section 4.

Through the paper $\<\cdot,\cdot \>$ shall denote the usual scalar product in the Euclidean space.

\section{The main result: a general uniqueness theorem}
%We have the following generalization of Theorem \ref{cor-cid}.
% In order to state it we consider the usual coordinate projections $\pi_1,\dots,\pi_{n+1}:\bR^{n+1}\to\bR$ and their cartesian products $\pi_{j_1}\times\cdots\times\pi_{j_k}:\bR^{n+1}\to\bR^k$, $k\in\bN$, $j_1,\dots,j_k\in\{1,\dots,n+1\}$.

\begin{dfn}Let $p_0\in \bR^{n+1}$. Assume there exist open subsets $V\subset\bR^n$, $U\subset\bR^{n+1}$, an open interval $J\subset \bR$ with $0\in J$ and a family of differentiable functions $\{g_s:V\to U\}_{s\in J}$ such that $g_0(0)=p_0\in U$ and $\Phi:(s,y)\in J\times V \to g_s(y)\in U$ is a diffeomorphism. Then we say $\{g_s\}_{s\in J}$ is a \emph{local $n$-foliation of U at $p_0$}.

%If moreover,  $F\circ g_s$ is locally Lipschitz for every $s\in J$ and $(\Phi')^{-1}$ is Lipschitz continuous in a neighborhood of zero then we say $\{g_s\}_{s\in J}$ is a \emph{F-Lipschitz local $n$-foliation of U at $p_0$}.

\end{dfn}
\begin{rem}An observation regarding notation. If $\Phi:\bR^{n+1}\to\bR^{n+1}$ is a diffeomorphism, we denote by $\Phi'$ its derivative and by $\Phi^{-1}$ its inverse. Also, we write $(\Phi^{-1})'$ for the derivative of the inverse. Observe that $\Phi'$ takes values in $\cM_{n+1}(\bR)$ so, although we cannot consider the functional inverse of $\Phi'$, we can consider the inverse matrix, whenever it exists, of every $\Phi'(x)$ for $x\in\bR^{n+1}$. We denote this function by  $(\Phi')^{-1}$. Clearly, the chain rule implies that
\begin{displaymath}(\Phi')^{-1}(x)=(\Phi^{-1})'(\Phi(x)).\end{displaymath}
\end{rem}

The following is our main result.
\begin{thm}\label{thmgen} Let $U\subset\bR^{n+1}$, $V\subset\bR^{n}$ be open sets, $p_0\in U$, $F:U\subset\bR^{n+1}\to\bR^{n+1}$ a continuous function and $\{g_s:V\to U\}_{s\in J}$ a local $n$-foliation of $U$ at $p_0$ which defines the diffeomorphism $\Phi:J\times V\to U$. If the following assumptions hold,

\begin{itemize}
\item{(C1)}  Transversality condition: \begin{equation}\label{transcon}\<\(\frac{\partial\Phi_1^{-1}}{\partial z_1}(p_0),\dots,\frac{\partial\Phi_1^{-1}}{\partial z_{n+1}}(p_0)\),F(p_0)\>\ne0,\end{equation}
\item{(C2)} Lipschitz condition along the foliation:
$F\circ \Phi$ and $(\Phi')^{-1}$ are Lipschitz in a neighborhood of zero when fixing the first variable, 

\end{itemize}
then there exists $\a>0$ such that problem~\eqref{peq-aut} has a unique solution in $[t_0-\a,t_0+\a]$.

% $f(t_0,x_0)\not\perp\< Dg_0\>^{\perp}$.
\end{thm}

\begin{proof} Consider the change of coordinates \begin{align}\label{coc}z=(z_1,\dots,z_{n+1})=\Phi(s,y_1,\dots,y_n):=g_s(y_1,\dots,y_n).\end{align}
 Since $\{g_s\}_{s\in J}$ is a foliation, $\Phi$ is a diffeomorphism. Then, considering $y=(s,y_1,\dots,y_n)$, differentiating \eqref{coc} with respect to $t$ and taking into account equation \eqref{peq-aut},
 \begin{equation}\label{eqderiv}\frac{\dif z}{\dif t}=\Phi'(y)\frac{\dif y}{\dif t}=F(z)=(F\circ \Phi)(y).\end{equation}
Since $\Phi$ is a diffeomorphism, $\Phi'(y)$ is an invertible matrix for every $y$, so
\[\frac{\dif y}{\dif t}=\Phi'(y)^{-1}(F\circ \Phi)(y).\]
By definition of $g_s$, $\Phi(0)=p_0$, so we can consider the problem
\begin{equation}\label{redeq}\frac{\dif y}{\dif t}(t)=h(y),\ y(t_0)=0,\end{equation}
where
\begin{equation*}h(y)=\Phi'(y)^{-1}F(  \Phi(y)).\end{equation*}
Now, by (C2) we have that $h$ is the product of locally Lipschitz functions when fixing the first variable. Furthermore, if $e_1=(1,0,\dots,0)\in\bR^n$ and taking into account (C1),
\[h_1(0)=e_1^T\Phi'(0)^{-1}F(p_0)=e_1^T(\Phi^{-1})'(p_0)F(p_0)=\<\(\frac{\partial\Phi_1^{-1}}{\partial z_1}(p_0),\dots,\frac{\partial\Phi_1^{-1}}{\partial z_{n+1}}(p_0)\),F(p_0)\>\ne0.\]

Hence, we can apply Theorem \ref{cor-cid} to problem \eqref{redeq} and conclude that problem~\eqref{peq-aut} has, locally, a unique solution.
\end{proof}
%\begin{lem}Let $f,h:\bR\to\bR$ be continuous at $0$, $f(0)=h(0)=0$. Then
%\[\lim_{\e\to0}\frac{f(h(\e)\e)}{\e}=0.\]
%\end{lem}
%\begin{thm}Let $F$ (in the conditions of Theorem \ref{thmgen}) be Lipschitz with respect to a $n$-foliation $\{g_s\}$. Then $F$ is Lipchitz with respect to the $n$-foliation $\til g_s(y)=Dg_0(0)y+s\frac{\dif g_0}{\dif s}(0)$.
%\end{thm}
\begin{rem}
1) Condition \eqref{transcon} can be easily interpreted geometrically: the vector \begin{displaymath}\(\frac{\partial\Phi_1^{-1}}{\partial z_1}(p_0),\dots,\frac{\partial\Phi_1^{-1}}{\partial z_{n+1}}(p_0)\),\end{displaymath} is normal to the hypersurface given by  $g_0(V)$ at $p_0$. So, condition \eqref{transcon} means that the vector $F(p_0)$ is not tangent to that hypersurface, and therefore it is called the \emph{'transversality condition'}. \medbreak

\noindent 2) Notice that from \cite[Example 3.1]{Cid03} we know that if the transversality condition \eqref{transcon}  does not hold then the Lipschitz condition along the foliation, that is (C2), is not enough to ensure uniqueness. On the other hand, by \cite[Example 3.4]{Cid03} we also know that (C1) and a Lipschitz condition along a local (n-1)-foliation do not imply uniqueness. So, in some sense, conditions (C1) and (C2) are sharp.

\end{rem}
%\begin{rem}
%Theorem \ref{thmgen} generalizes the main result in \cite{DiNoSi14}. Consider
%\[\Phi (t,s):=(t+s \varphi (t),x_0+s \psi (t)).\]
%Observe that, for $t$ in a neigbourhood of $t_0$ fixed and $s_1,s_2\in\bR$, take \[k_1=\frac{\varphi(t)}{\varphi( t_3)}s_1, k_2=\frac{\varphi(t)}{\varphi( t_4)}s_2.\]
%Then, \begin{align*} & |f(t+s_1\varphi(t),\,x_0+s_1\psi(t))-f(t_3,\,x_3)|=|f( t_3+k_1\varphi(t_3),x_3+)-f(t_3,\,x_3)| \\  \le & |f(t+s_2\varphi(t),\,x_0+s_2\psi(t))-f(t,x_0)|+|f(t,x_0)-f(t+s_1\varphi(t),\,x_0+s_1\psi(t))| \\\le & |f(t+t-t+s_2\varphi(t),\,x_0+s_2\psi(t)-f(t,x_0)|+|f(t,x_0)-f(t+t-t+s_1\varphi(t),\,x_0+s_1\psi(t))|\end{align*}
%\end{rem}
Theorem \ref{thmgen} generalizes the main result in \cite{SiNoDi16}, where only foliations consisting of hyperplanes are considered. In the next example we show the limitations of linear (or affine) coordinate changes which are used in \cite{SiNoDi16}.
\begin{exa}\label{exa-change}  Let $F(x,y):=1+(y-x^2)^{\frac{2}{3}}$. Is there a linear change of coordinates $\Phi$ such that $F\circ\Phi$ is Lipschitz in a neighborhood of zero when fixing the first variable? The answer is no. Any linear change of variables $\Phi$ will be given by two linearly independent vectors $v, w\in\bR^2$ as $\Phi(z,t)=z w+t v$. If $F\circ\Phi$  is Lipschitz in a neighborhood of zero when fixing the first variable, that is, $z$, that implies that the directional derivative of $F$ at any point of the neighborhood in the direction of $v$, whenever it exists, is a lower bound for any Lipschitz constant.  To see that this cannot happen, take $S=\{(x,y)\in\bR^2\ : y=x^2\}$ and realize that $F$ is differentiable in $\bR^2\backslash S$, with
\[\nabla F(x,y)=\frac{2}{3}(y-x^2)^{-\frac{1}{3}}(-2x,1),\quad \mbox{for} \, \, (x,y)\in\bR^2\backslash S.\]
Let $v=(v_1,v_2)\in\bR^2$. The directional derivative of $F$ at $(x,y)$ in the direction of $v$ is
\[D_vF(x,y)=\<\nabla F(x,y),v\>=\frac{2}{3}(y-x^2)^{-\frac{1}{3}}(v_2-2v_1x),\quad \mbox{for} \, \, (x,y)\in\bR^2\backslash S.\]
Now consider a neighborhood $N$ of $0$. In particular, we can consider the points of the form $(x,y)=(\l,\l^2+\mu)\in N\backslash S$ for $\mu\ne0$ and $\l\in(-\e,\e)$, so
\[D_vF(x,y)=\frac{2}{3}\frac{v_2-2 \lambda  v_1}{ \mu ^{1/3}}.\]
This quantity is unbounded in $ N\backslash S$ unless the numerator is $0$ for every $\l\in(-\e,\e)$, but that means that $v=0$, so $v$ and $w$ cannot be linearly independent.
Hence, no linear change of coordinates $\Phi$ makes $F\circ\Phi$ Lipschitz in a neighborhood of zero when fixing the first variable.\par

\begin{figure}[htbp]
 \begin{center}
  \includegraphics[scale=0.5]{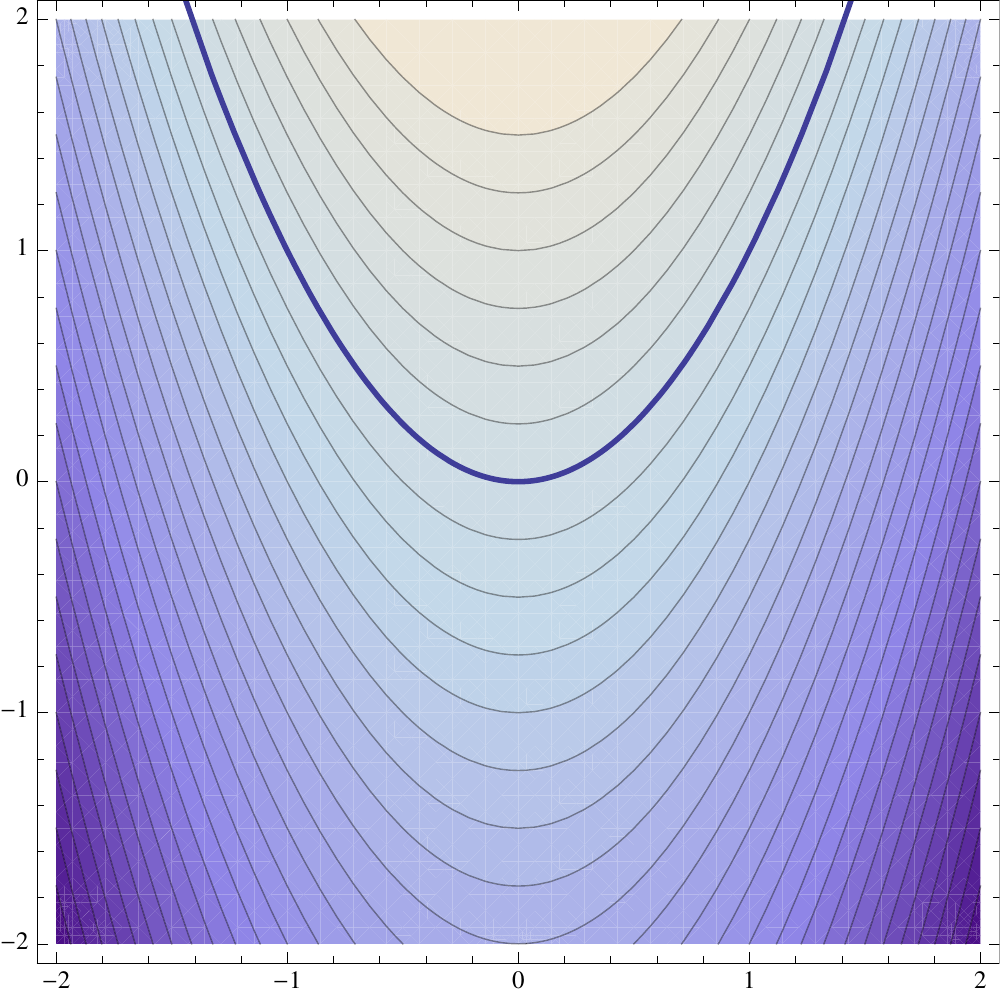}
  \end{center}
\end{figure}

Nevertheless, take $(x,y)=\Phi(z,t)=g_z(t)=(t,z+t^2)$. We have $\Phi^{-1}(x,y)=(y-x^2,x)$ and both are differentiable, so $\Phi$ is a diffeomorphism. Now,
$(F\circ\Phi)(z,t)=1+z^\frac{2}{3}$, which is clearly Lipschitz when fixing the first variable. 

In the figure you can see the parabolas $g_z(t)$ foliating the plane, where $g_0(t)$ is the thicker one.
\end{exa}

\begin{exa} With what we learned from Example \ref{exa-change}, it is easy to see that uniqueness for the scalar initial value problem
\begin{equation}\label{eqex1}x'(t)=1+(x(t)-t^2)^{\frac{2}{3}}, \quad x(0)=0,\end{equation}
can not be dealt with \cite[Theorem 2]{SiNoDi16} neither with \cite[Theorem 1]{DiNoSi14}. However, by using the local 1-foliation associated to diffeomorphism $\Phi$ given in Example \ref{exa-change}, it is easy to show that conditions (C1) and (C2) of Theorem \ref{thmgen} are satisfied. Therefore, we have the local uniqueness of solution.
\end{exa}

\section{Some results about Lipschitz functions}

We will now establish some properties of Lipschitz functions that will be useful for checking condition (C2) in Theorem \ref{thmgen}. Before that, consider the following Lemma.
\begin{lem}\label{leminv} Let $A,B,C\in\cM_n(\bR)$, $A$ and $C$ invertible. Then
\[ \|ABC\| \ge \frac{\|B\|}{\|A^{-1}\|  \|C^{-1}\|},\]
where $\|\cdot\|$ is the usual matrix norm.
\end{lem}
\begin{proof}It is enough to observe that
\[\|B\|=\|A^{-1}ABCC^{-1}\|\le\|A^{-1}\|\|ABC\|\|C^{-1}\|.\]
\end{proof}
\begin{lem}\label{lemeq}Let $U$ be an open subset of $\bR^n$ and $g:U\to GL_n(\bR)$.
\begin{enumerate}
\item If $g$ is locally Lipschitz and $g^{-1}$ (the inverse matrix function) is locally bounded, then $g^{-1}$  is locally Lipschitz.
\item If $g$ is locally Lipschitz when fixing the first variable and $g^{-1}$ is locally bounded, then $g^{-1}$ is locally Lipschitz when fixing the first variable.
\end{enumerate}
\end{lem}
\begin{proof}
1. Let $K$ be a compact subset of $U$, $k_1$ be a Lipschitz constant for $g$ in $K$ and $k_2$ a bound for $g^{-1}$ in $K$. Then, for $x,y\in K$, using Lemma \ref{leminv},
 \begin{align*}k_1\|x-y\| & \ge\|g(x)-g(y)\|=\|g(x)(g(y)^{-1}-g(x)^{-1})g(y)\|\ge\frac{\|g(y)^{-1}-g(x)^{-1}\|}{k_2^2}. \end{align*}
 Hence, $\|g(x)^{-1}-g(y)^{-1}\|\le k_1k_2^2\|x-y\|$ in $K$ and $g^{-1}$ is locally Lipschitz.\par
 2. We proceed as in 2. Let $K$ be a compact subset of $U$, $(t,x),(t,y)\in K$, $k_1$ be a Lipschitz constant for $g$ in $K$ when fixing $t$  and $k_2$ a bound for $g^{-1}$ in $K$. Then, 
  \begin{align*}k_1\|x-y\| & \ge\|g(t,x)-g(t,y)\|=\|g(t,x)(g(t,y)^{-1}-g(t,x)^{-1})g(t,y)\|\ge\frac{\|g(t,y)^{-1}-g(t,x)^{-1}\|}{k_2^2}. \end{align*}
  Hence, $\|g(t,x)^{-1}-g(t,y)^{-1}\|\le k_1k_2^2\|x-y\|$ and $g^{-1}$ is locally Lipschitz when fixing the first variable.\par
\end{proof}
\begin{cor}\label{coreq}Let $U$ be an open subset of $\bR^n$, $f:U\to f(U)\subset\bR^n$ be a diffeomorphism (notice that, in that case, $f':U\to GL_n(\bR)$).
\begin{enumerate}

\item If $f'$ is locally Lipschitz and $(f')^{-1}$ is locally bounded, then $(f')^{-1}$ is locally Lipschitz.

\item If $f'$ is locally Lipschitz and $(f')^{-1}$ is locally bounded, then $(f^{-1})'$ is locally Lipschitz.

\item If $f'$ is locally Lipschitz when fixing the first variable and $(f')^{-1}$ is locally bounded, then $(f')^{-1}$ is locally Lipschitz when fixing the first variable.

\end{enumerate}
\end{cor}
\begin{proof}

1. Just apply Lemma \ref{lemeq}.1 to $g=f'$.\par

2. Notice that \begin{displaymath}(f^{-1})'(x)=(f')^{-1}(f^{-1}(x)),\end{displaymath} and that  $(f')^{-1}$ is locally Lipschitz by the previous claim. On the other hand, 
since $f'$ is locally continuous we have that $f$ is locally a ${\cal C}^1$-diffeomorphism, and thus
$f^{-1}$ is locally Lipschitz. Therefore $(f^{-1})'$ is locally Lipschitz since it is the composition of two locally Lipschitz functions. \par

% Since $f$ is a diffeomorphism, we have in particular that $f^{-1}$ is locally Lipschitz.
% Let $K$ be a compact subset of $U$. Let $k_1$, $k_2$ be, respectively,  Lipschitz constants for $f'$ and $f^{-1}$ in $K$. Let $k_3$ be a bound of $(f^{-1})'$ in $K$.  Then, for $x,y\in f(U)$,
%\[\|f'(f^{-1}(x))-f'(f^{-1}(y))\|\le k_1\|f^{-1}(x)-f^{-1}(y)\|\le k_1k_2\|x-y\|.\]
%On the other hand, using Lemma \ref{leminv},
%\begin{align*} & \|f'(f^{-1}(x))-f'(f^{-1}(y))\|=\|[(f^{-1})'(x)]^{-1}-[(f^{-1})'(y)]^{-1}\| \\= & \|[(f^{-1})'(x)]^{-1}[(f^{-1})'(y)-(f^{-1})'(x)][(f^{-1})'(y)]^{-1}\|\\ \ge &\|(f^{-1})'(x)]\|^{-1}\|(f^{-1})'(y)-(f^{-1})'(x)\|\|(f^{-1})'(y)\|^{-1}.\end{align*}
%Hence,
%\[\|(f^{-1})'(y)-(f^{-1})'(x)\|\le k_1k_2k_3^2\|x-y\|,\]
%so $(f^{-1})'$ is locally Lipschitz.\par

3. Just apply Lemma \ref{lemeq}.2 to $g=f'$.
\end{proof}

\subsection{A modulus of continuity for Lipschitz functions along an hyperplane}

Let $U$ be an open subset of $\bR^{n+1}$, $p_0\in U$ and consider the tangent space of $U$ at $p$, which can be identified with $\bR^{n+1}$. Consider now the real Grassmannian $\Gr(n,n+1)$, that is, the manifold of hyperplanes of $\bR^{n+1}$. We know that $\Gr(n,n+1)\cong \Gr(1,n+1)=\bP^n$, that is, we can identify unequivocally each hyperplane with their perpendicular lines, which are elements of the projective space $\bP^n$.\par
\begin{dfn}
Consider $B_{n+1}(p,\d)\subset\bR^{n+1}$ to be the open ball of center $p$ and radius $\d$. Then, for a function $F:U\to\bR^{n+1}$ and every $p\in U$, $v\in\bP^n$ and $\d\in\bR^+$ we define the \textit{modulus of continuity} 
\[\omega_F(p,v,\d):=\sup_{\substack{x,y\in B_{n+1}(p,\d)\\ x-p,y-p\perp v\\x\ne y}}\frac{\|F(x)-F(y)\|}{\|x-y\|}\in[0,+\infty].\]
We also define
\[\omega_F(p,v):=\lim_{\d\to0} \omega_F(p,v,\d)=\lim_{\d\to0}\sup_{\substack{x,y\in B_{n+1}(p,\d)\\ x-p,y-p\perp v\\x\ne y}}\frac{\|F(x)-F(y)\|}{\|x-y\|}=\lils{\substack{(x,y)\to(p,p)\\x-p,y-p\perp v\\x\ne y}}\frac{\|F(x)-F(y)\|}{\|x-y\|}\in[0,+\infty].\]
\end{dfn}
\begin{rem}
If $\omega_F(p,v)<+\infty$, then there exist $\d,\e\in\bR^+$ such that
\[\|F(x)-F(y)\|\le(\omega_F(p,v)+\e)\|x-y\|,\ x,y\in B_{n+1}(p,\d),\ x-p,y-p\perp v.\]
Equivalently,
\[\|F(x+p)-F(y+p)\|\le(\omega_F(p,v)+\e)\|x-y\|,\ x,y\in B_{n+1}(0,\d),\ x,y\perp v.\]
Let $A$ be a orthonormal matrix such that its first column is parallel to $v$. In that case, 
since $A$ is orthogonal, $x\perp e_1$ implies that $Ax \perp v$.Then,
\begin{equation*}\|F(Ax+p)-F(Ay+p)\|\le(\omega_F(p,v)+\e)\|A(x-y)\|,\ x,y\in B_{n+1}(0,\d),\ x,y\perp e_1.\end{equation*}
That is, taking into account that $\|A\|=1$,
\begin{equation*}\|F(A(0,x)+p)-F(A(0,y)+p)\|\le(\omega_F(p,v)+\e)\|x-y\|,\ x,y\in B_{n}(0,\d).\end{equation*}
Hence, if $\phi(x)=Ax+p$ then $F\circ \phi$ is locally Lipschitz in an neighborhood of the origin \textit{when the first variable is equal  to zero}. 

%Nevertheless, this bound is not locally uniform in a neighbored of $p$ in the subspace $p+\<v\>^\perp$. This does happen when we have $\omega_F(p,v,\d)<+\infty$. In this case
%\[\|F(x)-F(y)\|\le\omega_F(p,v,\d)\|x-y\|,\ x,y\in B_{n+1}(p,\d),\ x-p,y-p\perp v.\]
%Furthermore, if $q-p\perp v$, we have that $x-p\perp v$ if and only if $x-q\perp v$, so for  $q\in B_{n+1}(p,\d)$, $q-p\perp v$,
%\[\|F(x)-F(y)\|\le\omega_F(p,v,\d)\|x-y\|,\ x,y\in B_{n+1}(q,\d-\|p-q\|),\ x-p,y-p\perp v.\]
%That is, for all $q\in B_{n+1}(p,\d)$, $q-p\perp v$, $\omega_F(q,v,\d-\|p-q\|)<+\infty$. If we consider the compact set $S:=p+<v>^\perp\cap B_{n+1}[p,\d/2]$, we have that $\{ B_{n+1}(q,\d-\|p-q\|)\cap S\}_{q\in S}$ is an open covering of $S$.
%Then we can extract a finite subcovering and conclude there is $\e\in\bR^+$ such that
%\begin{equation}\label{ortLip}\|F(x)-F(y)\|\le\omega_F(p,v,\d)\|x-y\|,\ x,y\in B_{n+1}(q,\e),\ x-p,y-p\perp v,\ q\in S.\end{equation}

\end{rem}

The following lemma illustrates the relation between the modulus of continuity $\omega_F$ and the partial derivatives of $F$.
\begin{lem}\label{lemreg}Assume $F$ is continuously differentiable in a neighborhood $N$ of $p$. Then
\[\omega_F(p,v)=\sup_{\substack{w\perp v\\\|w\|=1}}\|D_wF(p)\|.\]
\end{lem}
\begin{proof} Since $F'(z)$ is continuous at $p$, for $\{\e_n\}\to 0$ there exists $\{\d_n\}\to 0$ such that if $z\in  B_{n+1}(p,\d_n)$ and
$\|w\|=1$ then $\|F'(z)(w)\|\le \|F'(p)(w)\|+\e_n$. Hence, using the Mean Value Theorem,

 \begin{align*}  \sup_{\substack{x,y\in B_{n+1}(p,\d_n)\\ x-p,y-p\perp v\\x\ne y}}\frac{\|F(x)-F(y)\|}{\|x-y\|}\le & \sup_{\substack{x,y,z \in B_{n+1}(p,\d_n)\\ x-p,y-p \perp v\\x\ne y}}\frac{\|F'(z)(x-y)\|}{\|x-y\|} \le   \sup_{\substack{z \in B_{n+1}(p,\d_n)\\ u \in B_{n+1}(0,2 \d_n) \\ u \perp v\\u\ne 0}}\frac{\|F'(z)(u)\|}{\|u\|} \\
 = & \sup_{\substack{z \in B_{n+1}(p,\d_n)\\ d \in (0,2 \d_n) \\ w \perp v\\ \|w\|=1}}\frac{\|F'(z)(d w)\|}{\|d w\|}= \sup_{\substack{z \in B_{n+1}(p,\d_n) \\ w \perp v\\ \|w\|=1}} \|F'(z)(w)\| \\ 
 \le  &   \sup_{\substack{w \perp v\\ \|w\|=1}} \|F'(p)( w)\|+\e_n = \sup_{\substack{w \perp v\\ \|w\|=1}} \|D_w F(p)\|+\e_n .\end{align*}
  Then, taking the limit when $n\to\infty$, we obtain
\[\omega_F(p,v)\le \sup_{\substack{w\perp v\\\|w\|=1}}\|D_wF(p)\|.\]
   
On the other hand, assume $w\in\bS^n$ and $w\perp v$. Then $F(p+tw)=F(p)+t(D_wF(p)+g(t))$ where $g$ is continuous and $\lim_{t\to0}g(t)=0$. Therefore,
\[\|D_wF(p)\|=\left\|\frac{F(p+tw)-F(p)}{t}-g(t)\right\|\le \sup_{\substack{x,y\in B_{n+1}(p,t)\\ x-p,y-p\perp v\\x\ne y}}\left[\frac{\|F(x)-F(y)\|}{\|x-y\|}+|g(t)|\right].\]
Taking the limit when $t$ tends to zero, $\|D_wF(p)\|\le\omega_F(p,v)$, which ends the proof.
\end{proof}

\begin{rem} This definition of the modulus of continuity $\omega_F(\cdot,\cdot)$ is somewhat similar to the definition of strong absolute differentiation which appears in \cite[expression (1)]{ChIn}:

\noindent Let $(X,d_X)$ and $(Y,d_Y)$ be two metric spaces and consider $F:X\to Y$ and $p\in X$. We say $F$ is \emph{strongly absolutely differentiable at $p$} if and only if the following limit exists:
\[F^{|\prime|}(p):=\lim_{\substack{(x,y)\to(p,p)\\ x\ne y}}\frac{d_Y(F(x),F(y))}{d_X(x,y)}.\]

However, notice that there some important differences between $\omega_F(\cdot,\cdot)$ and $F^{|\prime|}$ when $X=\bR^n$ and $Y=\bR^m$. First, since $\omega(\cdot,\cdot)$ is defined with a supremmum, $\omega(\cdot,\cdot)$ is well defined in more cases than $F^{|\prime|}$. Also, in the definition of $\omega_F(\cdot,v)$, we are avoiding the direction of a certain vector $v$. This means that, while strong absolute differentiation implies continuity at the point (see \cite[Theorem 3.1]{ChIn}), $\omega(\cdot,\cdot)$ does not.\par
Regarding the similarities, when the partial derivatives of $F$ exist, $F^{|\prime|}=\|\sum_{k=1}^n\frac{\partial F}{\partial x_k}\|$ (see \cite[Theorem 3.6]{ChIn}).
\end{rem}
\begin{exa}\label{examc}Consider again $F(x,y):=1+(y-x^2)^{\frac{2}{3}}$ and $S=\{(x,y)\in\bR^2\ : y=x^2\}$. As was stated in Example \ref{exa-change}, we have that $F|_{\bR^2\backslash S}\in\cC^\infty(\bR^2\backslash S)$ and
\[\nabla F(x,y)=\frac{2}{3}(y-x^2)^{-\frac{1}{3}}(-2x,1),\quad \mbox{for}\ (x,y)\in\bR^2\backslash S.\]
Therefore, $\omega(p,v)<+\infty$ for every $(p,v)\in (\bR^2\backslash S)\times\bP^1$.\par  
On the other hand, for $p=(x_0,x_0^2)\in S$ and $v=(v_1:v_2)\in\bP^1$, if $x=(x_1,y_1)-p\perp v$ then $x=\l(-v_2,v_1)+p$ for some $\l\in \bR$. Analogously, we take $y=\mu(-v_2,v_1)+p$ for some $\mu\in \bR$. Hence, 
\begin{align*}\omega_F(p,v)= & \lils{\substack{(x,y)\to(p,p)\\x-p,y-p\perp v\\x\ne y}}\frac{\|F(x)-F(y)\|}{\|x-y\|}=\lils{\substack{(\l,\mu)\to(0,0)\\\l\ne \mu}}\frac{|F(\l(-v_2,v_1)+p)-F(\mu(-v_2,v_1)+p)|}{\|(\l-\mu)(-v_2,v_1)\|} \\= &
\lils{\substack{(\l,\mu)\to(0,0)\\ \l\ne \mu}}\frac{|[\l(2x_0v_2+v_1)-\l^2v_2^2]^\frac{2}{3}-[\mu(2x_0v_2+v_1)-\mu^2v_2^2]^\frac{2}{3}|}{|\l-\mu|} \end{align*}
We now can consider two cases: $(v_1:v_2)=(-2x_0:1)$ and $(v_1:v_2)\ne(-2x_0:1)$. In the first case, taking into account that $z^2+z+1\ge 3/4$ for every $z\in\bR$,
\begin{align*}\omega_F(p,v)= & \lils{\substack{(\l,\mu)\to(0,0)\\\l\ne \mu}} \frac{|(-\l^2v_2^2)^\frac{2}{3}-(-\mu^2v_2^2)^\frac{2}{3}|}{|\l-\mu|}=\lils{\substack{(\l,\mu)\to(0,0)\\\l\ne \mu}} \frac{|\mu^\frac{4}{3}-\l^\frac{4}{3}||v_2|^\frac{2}{3}}{|\l-\mu|} \\= & |v_2|^\frac{2}{3}\lils{\substack{(\l,\mu)\to(0,0)\\\l\ne \mu}} \left|\mu^\frac{1}{3}+\frac{\l}{\mu^\frac{2}{3}+\mu^\frac{1}{3}\l^\frac{1}{3}+\l^\frac{2}{3}}\right|=|v_2|^\frac{2}{3}\lils{\substack{(\l,\mu)\to(0,0)\\\l\ne \mu}} \left|\mu^\frac{1}{3}+\l^\frac{1}{3}\frac{1}{\(\frac{\mu}{\l}\)^\frac{2}{3}+\(\frac{\mu}{\l}\)^\frac{1}{3}+1}\right| \\ \le & |v_2|^\frac{2}{3}\lils{\substack{(\l,\mu)\to(0,0)\\\l\ne \mu}} \left|\mu^\frac{1}{3}+\frac{4}{3}\l^\frac{1}{3}\right|=0. \end{align*}
Observe that in this deduction we have assumed $\l\ne0$. It is clear that, when $\l=0$, the limit is zero as well.\par
In the case $(v_1:v_2)\ne(-2x_0:1)$ the quotient inside the limit is not bounded and $\omega_F(p,v)=+\infty$. Therefore,
\[\omega_F^{-1}([0,+\infty))=(\bR^2\backslash S)\times\bP^1\cup\{((x,x^2),(-2x:1))\in\bR^2\times\bP^1\ :\ x\in\bR\}.\]
%From Example \ref{exa-change}, it is clear that $\omega^{-1}(p,v,\d)([0,+\infty))=(\bR^2\backslash S)\times\bP^1$ for every $\d\in\bR^+$.
\end{exa}

\section{Sufficient conditions ensuring a Lipschitz condition along a foliation}

The next Lemma is a key ingredient in the main result of this section. It gives an alternative expression to the rotation matrix provided by the Rodrigues' Rotation Formula and generalizes it for $n$-dimensional vector spaces.
\begin{lem}[Codesido's Rotation Formula]\label{crf} Let $x,y\in\bR^{n+1}$ and define $K_x^y\in\cM_{n+1}(\bR)$ as
\[K_x^y :=yx^T-xy^T.\]
Now, let $u,v\in \bS^n$, $v\not= -u$, and define $R_u^v\in\cM_{n+1}(\bR)$ as
\begin{equation}\label{Codfor}R_u^v:=\Id+K_u^v+\frac{1}{1+\<u,v\>}(K_u^v)^2,\end{equation}
where $\Id$ is the identity matrix of order $n+1$.

Then,  $R_u^v\in \operatorname{SO}(n+1)$ and $R_u^vu=v$, that is, $R_u^v$ is a rotation in $\bR^{n+1}$ that sends the unitary vector  $u$ to $v$. Furthermore, the function $R:\{(u,v)\in\bS^n\times\bS^n\ :\ u\ne-v\}\to \operatorname{SO}(n+1)$,  defined by $R(u,v):=R_u^v$, is analytic.
\end{lem}
\begin{proof} First, we show that 
 $R_u^v\in\operatorname{O}(n+1)$, that is, $(R_u^v)^T=(R_u^v)^{-1}$. Observe that $(K_u^v)^T=-K_u^v$ and so $[(K_u^v)^2]^T=(K_u^v)^2$. That is, $(R_u^v)^T= \Id-K_u^v+\frac{1}{1+\<u,v\>}(K_u^v)^2$.
 Therefore,
 \begin{align*}(R_u^v)^TR_u^v= & \left[\Id-K_u^v+\frac{1}{1+\<u,v\>}(K_u^v)^2\right]\left[\Id+K_u^v+\frac{1}{1+\<u,v\>}(K_u^v)^2\right] \\ = & \Id+\frac{1-\<u,v\>}{1+\<u,v\>}(K_u^v)^2+\frac{1}{(1+\<u,v\>)^2}(K_u^v)^4.\end{align*}
 Now,
 \begin{align*}(K_u^v)^2= & (vu^T-uv^T)^2=vu^Tvu^T+uv^Tuv^T-vu^Tuv^T-uv^Tvu^T=\<u,v\>(vu^T+uv^T)-(vv^T+uu^T),\\ (K_u^v)^4= & \left[\<u,v\>(vu^T+uv^T)-(vv^T+uu^T)\right]^2=\(\<u,v\>^2-1\)(K_u^v)^2.\end{align*}
 
Therefore,
\begin{align*}(R_u^v)^TR_u^v= & \Id+\frac{1-\<u,v\>}{1+\<u,v\>}(K_u^v)^2-\frac{1-\<u,v\>^2}{(1+\<u,v\>)^2}(K_u^v)^2=\Id.\end{align*}
Clearly, $R_u^v$ is  analytic on $S=\{(u,v)\in\bS^n\times\bS^n\ :\ u\ne-v\}$ and so is the determinant function. Now, we are going to prove that $S$ is a connected set: firstly, define 
the linear subspaces
\begin{displaymath}V_1:=\{z \in \bR^{2n+2} : \ z_i=-z_{n+1+i}, \quad i=1,2,\ldots n+1\},\end{displaymath}
\begin{displaymath}V_2:=\{z \in \bR^{2n+2} : \ z_i=0, \quad i=1,2,\ldots n+1\},\end{displaymath}
\begin{displaymath}V_3:=\{z \in \bR^{2n+2} : \ z_{n+1+i}=0, \quad i=1,2,\ldots n+1\},\end{displaymath}
and note that  $\codim(V_i)=n+1\ge 2$ for all $i\in \{1,2,3\}$. Then, it is know that  $X:=\bR^{n+1}\setminus (V_1 \cup V_2\cup V_3)$ is connected, see \cite[Chapter V, Problem 5]{Dieu69}, and since the projection $\pi: X \to S$ defined as
 \begin{displaymath}\pi(z)=\left(\frac{(z_1,z_2,\ldots,z_{n+1})}{\|(z_1z_2,\ldots,z_{n+1})\|},\frac{(z_{n+2},z_{n+3},\ldots,z_{2n+2})}{\|(z_{n+2},z_{n+3},\ldots,z_{2n+2})\|} \right),\end{displaymath}
  is continuous and onto, we have that $S$ is connected too.
Therefore, $|R_u^v|$ is continuous on the connected set $S$ and takes values in $\{-1,1\}$, so $|R_u^v|$ is constant. Since $|R_u^u|=|\Id|=1$  we have that $|R_u^v|=1$ on $S$, that is, $R_u^v\in\operatorname{SO}(n+1)$.\par
Last, observe that
\begin{align*}R_u^vu & =u+(vu^T-uv^T)u+\frac{\<u,v\>(vu^T+uv^T)u-(vv^T+uu^T)u}{1+\<u,v\>} \\ &=u+v-uv^Tu+\frac{\<u,v\>(v+uv^Tu)-(vv^Tu+u)}{1+\<u,v\>} \\ & =v+\frac{\<u,v\>(v+uv^Tu+u-uv^Tu)-(vv^Tu+u)+u-uv^Tu}{1+\<u,v\>}=v+\frac{\<u,v\>(v+u)-vv^Tu-uv^Tu}{1+\<u,v\>}\\ &=v+\frac{\<u,v\>(v+u)-\<u,v\>v-\<u,v\>u}{1+\<u,v\>}=v.\end{align*}
\end{proof}

\begin{rem} For $n=1$ the function $R$ admits a continuous extension to $\bS^1\times\bS^1$. Indeed, let us consider $u,v\in \bS^1$, $v\not=-u$. Then  $u=(\cos(\a),\sin(\a))$ and $v=(\cos(\b),\sin(\b))$ for some $\a,\b \in \bR$, with $\b\not= \a+(2k+1)\pi$, $k\in \bZ$. Now, a direct computation shows that 
\begin{displaymath}R_{u}^v=\left(\begin{array}{ccc} \cos(\a-\b)  & \sin(\a-\b)  \\ -\sin(\a-\b) & \cos(\a-\b) \end{array}\right).\end{displaymath}
Therefore, 
\begin{displaymath}\lim_{v\to -u} R_{u}^v =\lim_{\beta \to \alpha+\pi}  \left(\begin{array}{ccc} \cos(\a-\b)  & \sin(\a-\b)  \\ -\sin(\a-\b) & \cos(\a-\b) \end{array}\right)= \left(\begin{array}{cc}-1 & 0 \\0 & -1\end{array}\right).\end{displaymath}

However, for $n\ge 2$ the function $R$ does not admit a continuous extension to $\bS^n\times\bS^n$. To see this, consider $u\in\bS^n$, $w\in\bR^{n+1}$, $w\perp u$, $w\not=0$ and define $v(w)=(w-u)/\|w-u\|$. Observe that $v(w)\in \bS^n$, $v(w)\not= -u$, $\displaystyle\lim_{\|w\|\to 0} v(w)=-u$ and 
\[K_u^{v(w)}=\frac{1}{\|w-u\|}K_u^w.\]
Hence,
\begin{align*}R_u^{v(w)} & =\Id+\frac{1}{\|w-u\|}K_u^w+\frac{\| w-u\|}{\| w-u\|+\<u,w\>-1}\frac{1}{\| w-u\|^2}(K_u^w)^2 \\ &
=\Id+\frac{1}{\|w-u\|}K_u^w+\frac{-ww^T-\|w\|^2uu^T}{\| w-u\|(\| w-u\|-1)}.
\end{align*}

Now, consider $\bar{w}\perp u$ with $\|\bar{w}\|=1$.Therefore, if it exists,
\[\lim_{v \to -u}R_u^{v}=\lim_{t \to 0}R_u^{ v(t \bar{w})}=\Id+\lim_{t \to0}\frac{-t^2 (\bar{w}\bar{w}^T-uu^T)}{\sqrt{t^2+1}(\sqrt{t^2+1}-1)}=\Id-2 (\bar{w}\bar{w}^T-uu^T).\]
But  in $\bR^{n+1}$, with $n\ge 2$, there exist at least two independent unitary vectors $\bar{w}_1$ and $\bar{w}_2$ in $\<u\>^{\perp}$, each of them leading to a different value of the right-hand side of the previous expression. Hence, the $\displaystyle \lim_{v \to -u}R_u^{v}$ does not exist and thus $R$ can not be continuously extended to $\bS^n\times\bS^n$.
\end{rem}

The following is the main result in this section and gives sufficient conditions for the existence of a $n$-foliation which allows $F$ to satisfy condition (C2) in Theorem \ref{thmgen}. 
\begin{thm}\label{mainthm}Let $U$ be an open subset of $\bR^{n+1}$, $p_0\in U$ and $F:U\to\bR^{n+1}$ continuous. Assume there exists an open interval $J$ with $0\in J$ and a simple path $\c=(\c_1,\c_2)\in\cC^1(J, U\times\bP^n)$ such that the following conditions hold:
\begin{itemize}
\item[(i)] $\c_1(0)=p_0.$
\item[(ii)] There exist $\d,M\in\bR^+$, such that $\omega_F(\c_1(t),\c_2(t),\d)<M$ for all $t\in J$.
\item[(iii)] $\c_1'(0) \not\perp \c_2(0)$.
\end{itemize}

Then, there exists  an open neighborhood of zero $\hat{U}\subset U\subset \bR^{n+1}$ such that $\Phi(s,y)$ is
a local $n$-foliation of $\hat{U}$. Moreover, $F\circ \Phi$ and $(\Phi')^{-1}$ are Lipschitz in a neighborhood of zero when fixing the first variable.
\end{thm}

\begin{proof}  Assume, without loss of generality, that $\c_1$ is parameterized by arc length, that is, $\|\c_1'(t)\|=1$ for all $t\in J$. Consider $\bS^n$ as covering space of $\bP^n$ with the usual projection $\pi:\bS^n\to\bP^n$. Take $v_0\in\pi^{-1}(\c_2(0))$, such that $v_0\not= -e_1$ where $e_1=(1,0,\dots,0)\in\bR^{n+1}$, and consider the lift $\til\c=(\c_1,\til\c_2):J\to V\times\bS^n$ of $\c$ such that $\til \c(0)=(p_0,v_0)$. 

%
%Now, since $\til \c_1$ is differentiable it is, in a neighborhood of $0$, an embedding of a real $1$-manifold (that is, diffeomorphic to an open interval neighborhood of zero), therefore, by the Theorem of the Rank (see for instance \cite[Therem 1.11]{Gadea}), there exists a coordinate system $\til\Phi:\til U\subset\bR^n\to W\subset\bR^{n+1}$, where $\til U$ is a neighborhood of zero such that $\til\Phi\circ\c_1(t)=(t,0,\dots,0)$ in a neighborhood of zero $\til J$ where this composition is well defined. All the same, the differentiable map $\Phi$ induces a the map $D\til\Phi:T\til U\to TW$ between the tangent spaces. Since $\bS^n\subset\bR^{n+1}\simeq T_x\til U$, $x\in \til U$, we can consider the differentiable path $\hat\c=(\hat\c_1,\hat\c_2)$ given by
%\begin{center}\begin{tikzcd}[row sep=tiny]
%\til J \arrow{r}{\hat\c} & V\times \bS^n\subset TW\\
% t \arrow[mapsto]{r} & \(\til\Phi\circ\c_1(t),\dfrac{{\til\Phi}^*\circ\til \c_2(t)}{|{\til\Phi}^*\circ\til\c_2(t)|}\)
%\end{tikzcd}
%\end{center}

Now, $\til\c_2$ is continuous, and $\<e_1,\til \c_2(0)\>=\<e_1,v_0\>\not=-1$ so we can consider an open interval $\til J \subset J$ where $\<e_1,\til \c_2(s)\>\ne-1$ (that is,  $\til \c_2(s)\ne-e_1$) for $s\in\til J$. Since $\til \c$ is differentiable and $\|\til \c_2(s)\|=1$ for every $s\in \til J$, we can consider the continuously differentiable function
\begin{center}\begin{tikzcd}[row sep=tiny]
\til J \arrow{r}{A} & \operatorname{SO}(n+1)\\
 s \arrow[mapsto]{r} & A(s):=R_{e_1}^{\til\c_2(s)}
\end{tikzcd}
\end{center}
where $R_u^v$ is defined as in Lemma \ref{crf}. Observe that denoting by $a_j(s)$ the columns of $A(s)$, that is,
\begin{displaymath}A(s)=\left(\begin{array}{c|c|c|c}a_1(s) & a_2(s) & \ldots & a_{n+1}(s)\end{array}\right),\end{displaymath}
we have that $a_1(s)=\til\c_2(s)$ and $\{ a_2(s),a_3(s),\ldots, a_{n+1}(s)\}$ is an orthonormal basis of $\til\c_2(s)^{\perp}$, (remember that $A(s) e_1=\til\c_2(s)$ and that $A(s)$ is an orthogonal matrix).

Now, we can define the differentiable function
$\Phi: \til J\times  \bR^{n}\to \bR^{n+1}$ given by
\begin{displaymath}\Phi(s,y):=\c_1(s)+A(s)(0,y).\end{displaymath}

\noindent {\it Claim 1. $g_s(y):=\Phi(s,y)$ is a local $n$-foliation.}

We easily compute
\begin{displaymath}\frac{\partial \Phi}{\partial s}(s,y)=\c_1'(s)+A'(s) (0,y),\end{displaymath}
\begin{displaymath}\frac{\partial \Phi}{\partial y}(s,y)=\left(\begin{array}{c|c|c|c}a_2(s) & a_3(s) & \ldots & a_{n+1}(s)\end{array}\right).\end{displaymath}
So 
\begin{displaymath}\Phi'(0,0)=\left(\begin{array}{c|c|c|c|c}\c_1'(0) & a_2(0) & a_3(0) & \ldots & a_{n+1}(0)\end{array}\right),\end{displaymath}
and since, by (iii), $\c_1'(0)\not\perp \til \c_2(0)=a_1(0)$ we have
\begin{displaymath}J_{\Phi}(0,0)=|\Phi'(0,0)|\not=0.\end{displaymath} 
Then, by the inverse function theorem there exist open sets $\hat{J}\subset \til J ,\hat{V}\subset V$ and $\hat{U}\subset U$ such that 
$\hat{J}\times\hat{V}$ contains the origin and $\Phi:\hat{J}\times\hat{V}\to\hat{U}$ is a 
diffeomorphism. Moreover, by (i), $ \Phi(0,0)=p_0$, so $\Phi(s,y)$,
a local $n$-foliation of $\hat{U}$.

\medbreak

\noindent {\it Claim 2. $F\circ \Phi$ is Lipschitz continuous in a neighborhood of zero when fixing the first variable.}

Notice that, by construction, $\Phi(s,y)- \c_1(s)\in\<\til\c_2(s)\>^{\perp}$. Now, condition (ii) implies that
\begin{align*} & \|F\circ\Phi(s,y_1)-F\circ\Phi(s,y_2)\|  =  \|F(\c_1(s)+A(s)(0,y_1))-F(\c_1(s)+A(s)(0, y_2))\| \\ \le & \omega_F(\c_1(s),\c_2(s),\d)\|\c_1(s)+A(s)(0,y_1)-\c_1(s)+A(s)(0, y_2)\|\le M \sup_{s\in \hat{J}}\|A(s)\|  \|y_1-y_2\|.\end{align*}
for every $s\in \hat{J}$ and $y_1,y_2\in B_{n}\left(0,\displaystyle\frac{\d}{\sup_{s\in \hat{J}}\|A(s)\|}\right)$.\par

\noindent {\it Claim 3. $(\Phi')^{-1}$ is Lipschitz continuous in a neighborhood of zero when fixing the first variable.}\par
Fix $s\in \hat{J}$. We have that
\[\Phi'(s,y)=\left(\begin{array}{c|c|c|c|c}\c_1'(s)+A'(s)(0,y) & a_2(s) & a_3(s) & \ldots & a_{n+1}(s)\end{array}\right).\]
Then,
\[\|\Phi'(s,x)-\Phi'(s, y)\|\le  \sup_{s\in\hat{J}}\|A(s)\| \|x- y\|,\]
so $\Phi'$ is Lipschitz continuous in a neighborhood of zero when fixing $s$. 

On the other hand, $(\Phi')^{-1}$ is a continuous function, therefore locally bounded. Hence, by Corollary \ref{coreq}.3, $(\Phi')^{-1}$ is Lipschitz continuous in a neighborhood of zero when fixing the first variable.

\end{proof}

\renewcommand{\abstractname}{Acknowledgments}
\begin{abstract}The authors want to express their gratitude towards Prof. Santiago Codesido (Universit\'e de Gen\`eve, Switzerland) for suggesting the rotation matrix expression \eqref{Codfor} and several useful discussions.
\end{abstract}


\begin{thebibliography}{1}
\providecommand{\url}[1]{{#1}}
\providecommand{\urlprefix}{URL }
\expandafter\ifx\csname urlstyle\endcsname\relax
  \providecommand{\doi}[1]{DOI~\discretionary{}{}{}#1}\else
  \providecommand{\doi}{DOI~\discretionary{}{}{}\begingroup
  \urlstyle{rm}\Url}\fi
  
  \bibitem{agalak} R. P. Agarwal and V. Lakshmikantham, {\it Uniqueness and Nonuniqueness Criteria
for Ordinary Differential Equations}, Series in Real Analysis,
Vol. 6, World Scientific, Singapore, 1993.

\bibitem{ChIn}
W. J. Charatonik and M. Insall, Absolute differentiation in metric spaces,
\newblock  \emph{Houston J. Math.} {\bf 38}  (2012) 1313--1328.

\bibitem{Cid03}
 J. {\'A} Cid, On uniqueness criteria for systems of ordinary
  differential equations, 
\newblock \emph{J. Math. Anal. Appl.} \textbf{281} (2003)
  264--275 
  
   \bibitem{cidpouso1} J. \'A. Cid and R. L. Pouso,  On first order ordinary differential
 equations with non-negative right-hand sides,
\newblock \emph{Nonlinear
 Anal.} {\bf 52} (2003) 1961--1977.  
 
  \bibitem{cp2} J. \'A. Cid, R. L\'opez Pouso, Does Lipschitz with respect to $x$ imply uniqueness for the differential equation $y'=f(x,y)$?, \textit{Amer. Math. Monthly} {\bf 116} (2009) 61--66.
 
 
  \bibitem{cl} E. A. Coddington, N. Levinson, \textit{Theory of Ordinary Differential Equations}, McGraw-Hill Book Company, New York-Toronto-London, 1955.
  
  \bibitem{cons}  A. Constantin, On Nagumo's theorem, {\it Proc. Japan Acad. Ser. A Math. Sci.} {\bf 86} (2010), 41--44.


\bibitem{DiNoSi14} J. Dibl{\'\i}k, C. Nowak and S. Siegmund, A general Lipschitz uniqueness
  criterion for scalar ordinary differential equations,
\newblock \emph{Electron. J. Qual. Theory Differ. Equ.}
 2014, 34, 6 pp.

\bibitem{ferreira} R. A. C. Ferreira, A Nagumo-type uniqueness result for an {\it n}th order differential equation, {\it Bull. London Math. Soc.} {\bf 45} (2013) 930--934.

 
 \bibitem{hartman} P. Hartman, {\it Ordinary Differential Equations}, Reprint of the second edition, Birkh\"auser, Boston, 1982.


\bibitem{Dieu69}  J. Dieudonn\'e, \emph{Foundations of Modern Analysis},  Academic Press, New York-London, 1969.

  \bibitem{mortici} C. Mortici, On the solvability of the Cauchy problem, 
\newblock \emph{Nieuw Arch. Wiskd. IV. Ser.}  {\bf 17}  (1999) 21--23.

  \bibitem{SiNoDi16}
S. Siegmund, C. Nowak and J. Dibl{\'\i}k, A generalized Picard-Lindel\"of theorem, 
\newblock  \emph{Electron. J. Qual. Theory Differ. Equ.} 2016,  28, 8 pp.


\bibitem{sn} H. Stettner and C. Nowak, Eine verallgemeinerte Lipschitzbedingung
als Eindeutigkeitskriterium bei gew\"ohnlichen Differentialgleichungen, 
\newblock \emph{Math. Nachr.}  {\bf 141} (1989), 33--35.

\end{thebibliography}
\end{document}